\documentclass[11pt,english]{article}
\usepackage[T1]{fontenc}
\usepackage[latin9]{inputenc}
\usepackage{geometry}
\usepackage{color}
\usepackage{refstyle}
\usepackage{mathtools}
\usepackage{amsmath}
\usepackage{amsthm}
\usepackage{amssymb}
\usepackage{stackrel}
\PassOptionsToPackage{normalem}{ulem}
\usepackage{ulem}
\usepackage{ytableau}

\newtheorem{thm}{Theorem}[section]

\newtheorem{prop}[thm]{Proposition}
\newtheorem{lem}[thm]{Lemma}
\newtheorem{cor}[thm]{Corollary}
\newtheorem*{thmnn}{Theorem}

\theoremstyle{definition}
\newtheorem{defn}[thm]{Definition}
\newtheorem{ex}[thm]{Example}
\newtheorem{rmk}[thm]{Remark}

\usepackage{tikz}
\usetikzlibrary{calc,decorations.pathreplacing,calligraphy,matrix}
\tikzset{Dotted/.style={     line width=\pgfkeysvalueof{/tikz/Young/dot size},     dash pattern=on 0.001\pgflinewidth off #1,line cap=round,     shorten <=#1},Dotted/.default=3pt,     vdots/.style={draw=none,path picture={      \draw let \p1=(path picture bounding box.north),         \p2=(path picture bounding box.south) in        [Dotted={(\y1-\y2)/4}]         (\p1) -- (\p2);     }},     cdots/.style={draw=none,path picture={      \draw let \p1=(path picture bounding box.east),         \p2=(path picture bounding box.west) in         [Dotted={(\x1-\x2)/4}]         (\p2) -- (\p1);     }},     Young tableau/.style={matrix of math nodes,nodes in empty cells,     nodes={draw,minimum size=\pgfkeysvalueof{/tikz/Young/cell size}, inner sep=0.5pt},     column sep=-\pgflinewidth,row sep=-\pgflinewidth},     Young/.cd,cell size/.initial=1.5em,     dot size/.initial=1.2pt }

\title{Polynomial Expressions for Symmetric Group Characters on Cycles}

\author{Tom Moshaiov and Shaul Zemel}

\begin{document}

\date{}

\maketitle{}

\begin{abstract}
In \cite{[CZ]}, Cohen and Zemel showed that for a partition $\lambda \vdash k$, the dimension of the irreducible representation of $S_{n}$ corresponding to the partition $(n-k,\lambda) \vdash n$ is a polynomial of degree $k$ in $n$, whose coefficients in the binomial basis count standard Young tableaux of shape $\lambda$ with special restrictions. In this paper, we generalize their results on the representation's dimension to character values on arbitrary cycles.
\end{abstract}

\section*{Introduction}

Given a partition $\mu \vdash n$, let $f^{\mu}$ denote the dimension of the irreducible representation of $S_{n}$ that is associated with $\mu$. There are several ways to obtain this number, one of them is given by the hook formula. It is also equal to the number of standard Young tableaux of shape $\mu$.

Now fix a partition $\lambda \vdash k$, and consider the case where $\mu$ is the partition $(n-k,\lambda) \vdash n$, when $n$ is large. It is well-known (see, e.g., \cite{[Ra]}) that this expression is an integer-valued polynomial in $n$, of degree $k$. One may therefore ask about the coefficients showing up in its expansion in a suitable basis for the space of polynomials.

As this polynomial is integer-valued, the natural basis is the binomial one. In \cite{[CZ]}, Cohen and the second author found an expression for the coefficients in the resulting expansion, as numbers counting standard Young tableaux of shape $\lambda$ that satisfy some restriction. More precisely, Theorem 3.7 of that reference, which appears as Theorem \ref{dimform} here, states that if we write $f^{(n-k,\lambda)}$, as a polynomial in $n$, as the sum $\sum_{h=0}^{k}(-1)^{h}a_{\lambda,h}\binom{n}{k-h}$, then $a_{\lambda,h}$ counts the number of standard Young tableaux of shape $\lambda$ in which the first $h$ numbers lie in the first column. In particular $a_{\lambda,h}=0$ if $h$ is larger than the length $\ell(\lambda)$ of $\lambda$, so that the sum can be taken up to $\ell(\lambda)$ only (see Proposition 2.3 and Equation (2) there).

\smallskip

The dimension $f^{\mu}$ is the value of the character $\chi^{\mu}$ associated with representation $\mathcal{S}^{\mu}$ corresponding to $\mu$ at the trivial element of $S_{n}$. We thus ask whether such polynomial expressions can be obtained for other values of these characters, and whether in an appropriate basis the coefficients also have such combinatorial interpretations. For this we recall that for different values of $n$ these characters are defined on elements (or conjugacy classes) of different groups. However as $S_{n-1}$ is naturally viewed as the stabilizer, in the natural action $S_{n}$ on the set of integers between 1 and $n$, of the last number $n$, we can use these identifications of elements in different symmetric groups (like we did with the trivial elements), and consider the resulting behavior.

Now, the fact that these character values $\chi^{(n-k,\lambda)}(\sigma)$ are, for fixed $\sigma$ of that sort, also polynomials of degree $k$ is classically known, as explained in Proposition \ref{polgen} or Corollary \ref{specs} below. However, we are interested in the coefficients of their expansions in an appropriate basis, and see whether they have a combinatorial meaning in terms of counting tableaux of some sort. In this paper we achieve this goal for permutations $\sigma$ consisting of a single cycle (plus an increasing number of fixed points). It turns out that a basis in which this works is not the binomial one as in \cite{[CZ]}, but its translation by a fixed parameter that depends on $\sigma$. For a cycle $\sigma$ of length $r$, where the translation is by $r$, we determine the coefficients in the polynomial arising from $\lambda \vdash k$ as an alternating sum of numbers, each of which describes the number of skew-tableaux of a shape that is obtained by removing some fixed partition from $\lambda$.

\smallskip

We now describe the results in more detail. For any fixed $s\in\mathbb{Z}$, the set $\big\{\binom{x-s}{m}\big\}_{m\geq0}$ forms a basis for the space of polynomials, which is, in fact, an integral basis for the additive group of integer-valued polynomials (this is done, for $s=0$, in Section 1 of \cite{[CZ]}, and the general case is similar or follows easily, but this result was known way earlier, of course). For $\sigma$ a cycle of length $r$, we consider the expansion of the polynomial $\chi^{(n-k,\lambda)}(\sigma)$ as $\sum_{h=0}^{k}(-1)^{h}b_{\lambda,h}^{(r)}\binom{n-r}{k-h}$, with $s=r$ (using the same indexation and sign $(-1)^{h}$ as in \cite{[CZ]}), and prove that the coefficient $b_{\lambda,h}^{(r)}$ is a combination of numbers counting skew-tableaux, with alternating signs.

More precisely, we define, for any $r$, a set of \emph{$r$-primary} partitions, which is a disjoint union $\bigcup_{h=0}^{\infty}\Gamma^{r}_{h}$ where $\Gamma^{r}_{h}$ is the set of those partitions $\nu \vdash h$ (see Definition \ref{rprim} below). Any such partition comes equipped with an \emph{$r$-sign} $\varepsilon^{r}_{\nu}\in\{\pm1\}$, and if we write $f^{\lambda\setminus\nu}$ for the number of skew-tableaux of shape $\lambda\setminus\nu$ wherever $\nu\subseteq\lambda$ (and set it to be 0 otherwise---see Definition \ref{skewtabs} below), then our main result, which is Theorem \ref{main} below, reads as follows.
\begin{thmnn}
The $h$th coefficient $b_{\lambda,h}^{(r)}$ in the expansion of $\chi^{(n-k,\lambda)}(\sigma)$ equals $\sum_{\nu\in\Gamma^{r}_{h}}\varepsilon^{r}_{\nu}f^{\lambda\setminus\nu}$.
\end{thmnn}
For example, in the case $r=2$ of transpositions, the 2-primary partitions are $(h)$ for $0 \leq h\leq3$ as well as $(3,1^{h-3})$ and $(2,2,1^{h-4})$ when $h\geq4$, with the latter carrying a negative sign while the signs of all the other ones are positive. Hence it we set $b_{\lambda,h}^{+}$ to be $f^{\lambda\setminus(h)}$ when $0 \leq h\leq3$ and $f^{\lambda\setminus(3,1^{h-3})}$ in case $h\geq4$, and $b_{\lambda,h}^{-}$ is $f^{\lambda\setminus(2,2,1^{h-4})}$ for $h\geq4$ and vanishes in case $0 \leq h\leq3$, then the result, which reduces to Theorem \ref{transthm} below, is that $b_{\lambda,h}^{(2)}=b_{\lambda,h}^{+}-b_{\lambda,h}^{-}$.

\smallskip

We indicate briefly how the proof works. In a manner similarly to \cite{[CZ]}, one easily verifies that both the coefficients $b_{\lambda,h}^{(r)}$ and the linear combinations of the $f^{\lambda\setminus\nu}$'s satisfy, when $h<k$, a close analogue of the Branching Rule (this is the case for permutation $\sigma$, not just a cycle---see Corollary \ref{coeffind} and Lemma \ref{recskew} below). Therefore the main issue is to establish, for $\lambda \vdash k$, the equality when $h=k$. For transpositions we can do so using an old formula of Frobenius from \cite{[F]}, which we cite as Proposition \ref{Frobenius} below. In general there are two ways to obtain these constant terms, either by applying the Murnaghan--Nakayama Rule, or using a formula that is based on cycle sizes as variables of a polynomial---see Proposition \ref{recpart} below. We present both constructions in this paper, and we also deduce an alternative expansion for the dimension using the value $r=1$ (see Remark \ref{r1alt} below), as well as show, in Corollary \ref{limr} that the coefficients ``stabilize as $r\to\infty$'' at the value $a_{\lambda,h}$ from \cite{[CZ]}.

We note that for a general permutation $\sigma$ both the proof and the form of the result is more complicated, as with the natural value of $s$ from Corollary \ref{specs}, the signs $\varepsilon^{r}_{\nu}$ arising from the basic cases in this inductive construction no longer have to be $\pm1$---even in the simplest non-cycle case where $\sigma$ is the product of two disjoint transpositions values with $\pm2$ already show up. This is why we focus only on cycles in this paper.

\smallskip

This paper is divided into 4 sections. Section \ref{ClassRes} cites some classical results about dimensions and characters and establishes some basic inductive properties. Section \ref{Trans} presents the proof of the main result for transpositions, using an old formula of Frobenius. In Section \ref{GenCyc} we give the proof for the general case, and Section \ref{AltPf} contains the second way to determine the constant coefficients that are required as the basic step for our construction.

We are thankful to Y. Roichman and R. Adin for numerous discussions and useful insights around these results.

\section{Some Classical Results and their Consequences \label{ClassRes}}

Recall that a \emph{partition} of $n$, denoted as $\mu \vdash n$, is a weakly decreasing sequence $\{\mu_{i}\}_{i=1}^{\ell}$ of positive integers such that $\sum_{i=1}^{\ell}\mu_{i}=n$. The number $\ell$ is the \emph{length} of the partition $\mu$, and is denoted by $\ell(\mu)$.

We identify, throughout the paper, a partition $\mu \vdash n$ with its Young diagram. Namely, if $\mu=\{\mu_{i}\}_{i=1}^{\ell}$ then the diagram consists of $\ell(\mu)$ left-aligned lines of boxes, with the $i$th line containing $\mu_{i}$ boxes. We will use the English notation, in which the indices of the rows increase as we go down the diagram. If a box $v$ with entries $(i,j)$ satisfies $1 \leq j\leq\mu_{i}$ then the box is contained in the diagram of $\mu$, and we simply write $v\in\mu$ and say that $v$ is contained in $\mu$.

\smallskip

If $v$ is a box that is contained in $\mu$, then the \emph{hook} $H_{v}$ that associated with $v$ is the set of boxes in $\mu$ that either lie below $v$ in the same column, or to the right of $v$ in the same row (including $v$ itself), and $h_{v}$ is the size $|H_{v}|$ of $H_{v}$. Then the \emph{Hook Formula}, given in, e.g., Theorem 3.10.2 of \cite{[S]}, reads as follows.
\begin{thm}
The dimension $f^{\mu}$ of the irreducible representation of $S_{n}$ corresponding to the partition $\mu \vdash n$ is given by
$n!\big/\prod_{v\in\mu}h_{v}$. \label{Hook}
\end{thm}

\smallskip

Recall that a box $v\in\mu$ is an \emph{internal corner} of $\mu$ if $h_{v}=1$. These are precisely the boxes that can be removed from $\mu$ and leave a complement that is itself the Young diagram of a partition, which we denote by $\mu-v$. We denote by $\operatorname{IC}(\mu)$ the set of internal corners of $\mu$, and when we identify $S_{n-1}$ as the stabilizer of $n$ in the natural action of $S_{n}$, the \emph{Branching Rule}, which appears as, e.g., Lemma 2.8.3 of \cite{[S]}, is the following result.
\begin{thm}
Let $\mathcal{S}^{\mu}$ denote the representation of $S_{n}$ that is associated with the partition $\mu \vdash n$, and denote its character by $\chi^{(\mu)}$. Then the restriction of $\mathcal{S}^{\mu}$ to $S_{n-1}$ decomposes as the direct sum $\bigoplus_{v\in\operatorname{IC}(\mu)}S^{\mu-v}$, and therefore for any element $\pi \in S_{n-1}$, viewed also as embedded inside $S_{n}$, we have the equality $\chi^{\mu}(\pi)=\sum_{v\in\operatorname{IC}(\mu)}\chi^{\mu-v}(\pi)$. \label{Branching}
\end{thm}
In particular, taking $\sigma$ to be the trivial element in Theorem \ref{Branching} yields $f^{\mu}=\sum_{v\in\operatorname{IC}(\mu)}f^{\mu-v}$.

\smallskip

In this paper, as in \cite{[CZ]}, we are interested in partitions of the form $\mu=(n-k,\lambda) \vdash n$ and the behavior as a function of $n$, where $\lambda \vdash k$ is a fixed partition. For the dimension, namely the value of the character at the trivial element, this behavior is given in Theorem 3.7 in \cite{[CZ]}, as described in the Introduction.
\begin{thm}
Let $\lambda \vdash k$ be a partition, and take any $n \geq k+\lambda_{1}$. Then $f^{(n-k,\lambda)}$ is a polynomial of degree $k$ as a function of $n$, and its expansion in the binomial basis for polynomials is \[f^{(n-k,\lambda)}=\sum_{h=0}^{\ell(\lambda)}(-1)^{h}a_{\lambda,h}\binom{n}{k-h},\] where $a_{\lambda,h}$ is the number of standard Young tableaux of shape $\lambda$ in which the numbers between 1 and $h$ are posited at the beginning of the first $h$ rows. \label{dimform}
\end{thm}

As mentioned in the Introduction, Theorem \ref{dimform} is the case where inside the representation, we take the character value of the trivial element. A key consequence of the Branching Rule is that in the appropriate sense, we get the polynomial behavior for every element.
\begin{prop}
Take an element $\sigma$ in $S_{m}$ for some $m$, and identify it with its image in $S_{n}$ for any $n \geq m$. Then, for any $\lambda \vdash k$ and any $n\geq\max\{k+\lambda_{1},m\}$, the expression $\chi^{(n-k,\lambda)}(\sigma)$ is the value at $n$ of a polynomial of degree $k$ with positive leading coefficient. \label{polgen}
\end{prop}

\begin{proof}
Write $g(n):=g^{\lambda}_{\sigma}(n)$ for the value $\chi^{(n-k,\lambda)}(\sigma)$, to emphasize that $n$ is the variable, and we argue by induction on $k$. If $k=0$ then $\lambda$ is empty, the representation associated with $(n-k,\lambda)=(n)$ is the trivial one, and the value of its character on $\sigma$ is the constant 1, hence a polynomial of degree $k=0$ having a positive leading coefficient, as desired.

We thus assume that $k>0$, and consider $n>\max\{k+\lambda_{1},m\}$. Then $\operatorname{IC}\big((n-k,\lambda)\big)$ consists of one box at the end of the first row (subtracting which yields $(n-1-k,\lambda) \vdash n-1$), and the shifting of the elements of $\operatorname{IC}(\lambda)$ one row below (and we identify these elements with $\operatorname{IC}(\lambda)$). Moreover, for every $v\in\operatorname{IC}(\lambda)$ we write the partition resulting from removing $v$ as $\big((n-1)-(k-1),\lambda-v\big)$, because $\lambda-v \vdash k-1$. Recalling that we work with the images of the same element $\sigma$, and subtracting the element associated with the inner corner in the first row, we obtain the equality
\begin{equation}
g^{\lambda}_{\sigma}(n)-g^{\lambda}_{\sigma}(n-1)=\sum_{v\in\operatorname{IC}(\lambda)}g^{\lambda-v}_{\sigma}(n-1). \label{indforpol}
\end{equation}

But as $\lambda-v \vdash k-1$ for every $v\in\operatorname{IC}(\lambda)$, the induction hypothesis implies that the right hand side of Equation (\ref{indforpol}) is obtained by substituting $n-1$ in a polynomial of degree $k-1$ having a positive leading coefficient, and is therefore the value at $n$ of a similar such polynomial. Therefore the sum over $v$ there is also a polynomial of the same sort (the positivity of the leading coefficient implies that there can be no cancelations, and the degree remains $k-1$).

It follows from that equation that $g^{\lambda}_{\sigma}$ is a function of $n$ whose discrete difference is a polynomial of degree $k-1$. But then there is a polynomial $h$ of degree $k$, with a positive leading coefficient, such that $h(x)-h(x-1)$ coincides with the polynomial arising from that equation. Since a polynomial is determined by this difference up to an additive constant, we deduce that $g^{\lambda}_{\sigma}$ can be written as this $h$ plus some constant, yielding the assertion. This proves the proposition.
\end{proof}

We saw in Theorem \ref{dimform} the expansion of the polynomial there in the binomial basis $\big\{\binom{x}{m}\big\}_{m\geq0}$ for the space of polynomials, which is a basis over $\mathbb{Z}$ for the integer-valued polynomials. This is natural, because dimensions, like all character values for the symmetric groups, are integers. More generally, we can fix some $s$, and then use $\big\{\binom{x-s}{m}\big\}_{m\geq0}$ as a basis for the polynomials (with the same integrality property). We then get the following consequence, generalizing Lemma 2.5 of \cite{[CZ]}.
\begin{cor}
Fix $s$ and $\sigma$, and write the polynomial $\chi^{(n-k,\lambda)}(\sigma)$ from Proposition \ref{polgen} as $\sum_{h=0}^{k}(-1)^{h}c_{\lambda,h}^{\sigma,s}\binom{n-s}{k-h}$. Then we have the equality
$c_{\lambda,h}^{\sigma,s}=\sum_{v\in\operatorname{IC}(\lambda)}c_{\lambda-v,h}^{\sigma,s}$ for every $0 \leq h<k$. \label{coeffind}
\end{cor}

\begin{proof}
Substituting the expansion of each summand in the right hand side of Equation (\ref{indforpol}) and changing the summation order produces $\sum_{h=0}^{k-1}(-1)^{h}\big[\sum_{v\in\operatorname{IC}(\lambda)}c_{\lambda-v,h}^{\sigma,s}\big]\binom{n-1-s}{k-1-h}$.

In the left hand side we get $\sum_{h=0}^{k}(-1)^{h}c_{\lambda,h}^{\sigma,s}\big[\binom{n-s}{k-h}-\binom{n-1-s}{k-h}\big]$. For $h=k$ the parentheses yield $1-1=0$ (the cancelation of the constants), and otherwise $k-h>0$ and they reduce to $\binom{n-1-s}{k-1-h}$. This expresses both sides, as polynomials of degree $k-1$ in $n$, using the basis $\big\{\binom{n-1-s}{k-1-h}\big\}_{h=0}^{k-1}$. As the latter is a basis, the corresponding coefficients on both sides must coincide, yielding the desired equality. This proves the corollary.
\end{proof}
Note that Corollary \ref{coeffind} implicitly assumes that $\lambda$ is not empty, for otherwise the assertion is true in an empty manner (since no $0 \leq h<k=0$ exists). It follows from Corollary \ref{coeffind} that if we know, for such $\lambda$, the expansion of $g^{\lambda-v}_{\sigma}$ for every $v\in\operatorname{IC}(\lambda)$, then we also know all the coefficients in the one for $g^{\lambda}_{\sigma}$ except the constant one associated with $h=k$. The sign $(-1)^{h}$ is not important for Corollary \ref{coeffind} and its proof, but it turns out to give the more convenient interpretation of the coefficients, as exemplified already in Theorem \ref{dimform}.

\smallskip

Theorem \ref{dimform} compares the coefficients $a_{\lambda,h}$ with the numbers counting certain standard Young tableaux, where we recall that $\operatorname{SYT}(\lambda)$ stands for the set of standard Young tableaux of shape $\lambda$. We will be using the following generalization, which is well-known from the literature.
\begin{defn}
Let $\lambda \vdash k$ and $\nu \vdash h$ be partitions. Let $\lambda\setminus\nu$ be the \emph{skew-shape} that is obtained by removing the boxes of $\nu$ from $\lambda$ in case $\nu\subseteq\lambda$ as shapes. We write $f^{\lambda\setminus\nu}$ for the number of standard Young tableaux of skew-shape $\lambda\setminus\nu$, and set it to be 0 in case $\nu\not\subseteq\lambda$. \label{skewtabs}
\end{defn}
By choosing an element $M_{\nu}\in\operatorname{SYT}(\nu)$, and identifying standard Young tableau of skew-shape $\lambda\setminus\nu$ with the tableaux obtained by increasing every entry by $h$, the number $f^{\lambda\setminus\nu}$ from Definition \ref{skewtabs} is easily shown to coincide with the number of elements of $\operatorname{SYT}(\lambda)$ whose restriction to $\nu$ yields $M_{\nu}$. Hence for $\nu=(1^{h})$, with the only possible choice of $M_{\nu}$, the numbers $a_{\lambda,h}$ from Theorem \ref{dimform} are the same as $f^{\lambda\setminus(1^{h})}$. Another equivalent definition for $f^{\lambda\setminus\nu}$ is as the number of monotone paths, in the Young lattice, from $\nu$ to $\lambda$.

Lemma 3.6 of \cite{[CZ]} has the following natural generalization.
\begin{lem}
For any $\nu$ and $\lambda\neq\nu$, we have the equality $f^{\lambda\setminus\nu}=\sum_{v\in\operatorname{IC}(\lambda)}f^{(\lambda-v)\setminus\nu}$. \label{recskew}
\end{lem}

\begin{proof}
If $\lambda$ does not contain $\nu$ then neither does $\lambda-v$ for any $v\in\operatorname{IC}(\lambda)$, and both sides vanish (this is valid also when $\lambda=\emptyset$, since then $\nu$ cannot be empty, as the right hand side being an empty sum). We thus assume that $\nu\subsetneq\lambda$, and recall that $f^{\lambda\setminus\nu}$ also counts the monotone paths from $\nu$ to $\lambda$.

But as $\nu\subsetneq\lambda$, every such path contains at least one step, meaning that there is a last element of the Young lattice which is reached just before $\lambda$. As the element lying just below $\lambda$ in that lattice are precisely the partitions $\{\lambda-v\;|\;v\in\operatorname{IC}(\lambda)\}$, and every path goes through precisely one of these partitions, we can decompose the set of such paths according to the union over $v\in\operatorname{IC}(\lambda)$ of the set of paths from $\nu$ to $\lambda$ that go through $\lambda-v$.

So fix such $v$, and then a path going through $\lambda-v$ is a concatenation of a path from $\nu$ to $\lambda-v$ and the one-step path from the latter to $\lambda$. As the latter parts is uniquely determined, the number of such paths coincides with the number of paths from $\nu$ to $\lambda-v$, which is $f^{(\lambda-v)\setminus\nu}$ by definition. The fact that $f^{\lambda\setminus\nu}$ is the size of a disjoint union over $v\in\operatorname{IC}(\lambda)$ of sets, and the size of the set corresponding to $v$ is $f^{(\lambda-v)\setminus\nu}$, yields the desired equality. This proves the lemma.
\end{proof}
One can choose a tableau $M_{\nu}\in\operatorname{SYT}(\nu)$ and use it for proving Lemma \ref{recskew} in the terminology of Remark 3.5 and Lemma 3.6 of \cite{[CZ]}, but the current formulation is shorter. Note that as in Lemma 2.5 of that reference, the condition that $\lambda\neq\nu$ is important, since when $\lambda=\nu$ we have a contribution of 1 to $f^{\lambda\setminus\nu}$ from the trivial path (or from $M_{\nu}$), while $f^{(\lambda-v)\setminus\nu}=0$ for every $v\in\operatorname{IC}(\lambda)$ and the equality does not hold.

\smallskip

A tool that is used in one of the proofs of the general case is the Murnaghan--Nakayama Rule. This rule, which generalizes the Branching Rule, is a recursive formula which allows the computation of the characters of the symmetric groups. Before giving the rule, we first recall the relevant terminology (following \cite{[Ro]}). To state it we recall that the \emph{boundary} of a partition (or Young diagram) $\lambda$ is the set of cells $(i,j)\in\lambda$ for which $(i+1,j+1)\not\in\lambda$. Given such $\lambda$, a \emph{skew-hook} (in $\lambda$) is a connected subset $\rho$ of the boundary of $\lambda$ such that $\lambda\setminus\rho$ is the Young diagram of a partition. Each such non-empty $\rho$ consists of boxes lying in several consecutive rows (with one or more boxes in each such row), and we denote by $\mathrm{ll}(\rho)$ the \emph{leg length} of $\rho$, which is one less than the number of these rows.

The Murnaghan--Nakayama Rule reads as follows (see, e.g., Theorem 4.10.2 of \cite{[S]}).
\begin{thm}
Let $\mu \vdash n$ be a partition, then for $r\geq1$ we denote by $\operatorname{SH}_{r}^{\mu}$ the set of skew-hooks of $\mu$ which consist of $r$ cells. Consider an element $\pi \in S_{n-r}$, and let $\eta \in S_{n}$ be the product of $\pi$ with a cycle $\sigma$ of length $r$ that is supported on the numbers between $n-r+1$ and $n$. Then we have the equality \[\chi^{\mu}(\eta)=\sum_{\rho\in\operatorname{SH}_{r}^{\mu}}(-1)^{\mathrm{ll}(\rho)}\chi^{\mu\setminus\rho}(\pi).\] In particular, when $\pi$ is trivial and $\eta=\sigma$ this expresses $\chi^{\mu}(\sigma)$ as $\sum_{\rho\in\operatorname{SH}_{r}^{\mu}}(-1)^{\mathrm{ll}(\rho)}f^{\mu\setminus\rho}$. \label{MNrule}
\end{thm}
As when $r=1$ in Theorem \ref{MNrule} we have $\operatorname{SH}_{1}^{\mu}=\operatorname{IC}(\mu)$, with all elements having leg length 0, and $\sigma$ is trivial (so that $\eta \in S_{n}$ is the natural image of $\pi \in S_{n-1}$), this reproduces Theorem \ref{Branching}.

For more background on the representation theory of symmetric groups, see, e.g., \cite{[S]}.

\section{The Case of Transpositions \label{Trans}}

We are interested in combinatorial expressions for the coefficients $c_{\lambda,h}^{\sigma,s}$ when $\sigma$ is a cycle, and $s$ is chosen appropriately. Theorem \ref{dimform} achieves this result for trivial $\sigma$, by taking $s=0$. Corollary \ref{coeffind} implies that we can determine the values of all the coefficients by induction, except for the new one with $h=k$ (which is the constant term in this expansion). This shortens the analysis of these coefficients considerably, but we need other tools for finding what the constant one is.

For general cycles, we will apply the Murnaghan--Nakayama Rule from Theorem \ref{MNrule}. However, in this section we consider the case of transpositions, as this case is the next simplest case after the trivial element considered in \cite{[CZ]}, and we also have an explicit, direct formula for the value of the characters.

Recall that if $\mu \vdash n$, then $\mu^{t}=\{\mu^{t}_{i}\}_{i=1}^{\mu_{1}}$ is the \emph{transpose partition}, whose diagram is obtained by reflecting that of $\mu$ along the diagonal, and in which $\mu^{t}_{j}=|\{1 \leq i\leq\ell(\mu)\;|\;\mu_{i} \geq j\}|$, so that in particular $\mu^{t}_{1}=\ell(\mu)$. The formula for the value of the character at a transposition is named after Frobenius who proved it in \cite{[F]}.
\begin{prop}
Assume that $n\geq2$, take $\mu \vdash n$, and let $\tau \in S_{n}$ be a transposition. Then
\[\chi^{\mu}(\tau)=\frac{f^{\mu}}{\binom{n}{2}}\sum_{i\geq1}\bigg[\binom{\mu_{i}}{2}-\binom{\mu^{t}_{i}}{2}\bigg].\] \label{Frobenius}
\end{prop}

\smallskip

When we apply Proposition \ref{Frobenius} to $\mu=(n-k,\lambda)$, Theorem \ref{dimform} shows that the expansion of $f^{\lambda}$ is based only on terms with $0 \leq h\leq\ell(\lambda)$. Hence if $\ell(\lambda) \leq k-2$ then only terms $\binom{n}{m}$ with $m\geq2$ show up in the expansion, and they cancel with the denominator $\binom{n}{2}$ from that proposition. We will keep the notation $\tau$ for a transposition, and get the following result.
\begin{prop}
Set $b_{\lambda,h}^{(2)}:=c_{\lambda,h}^{\tau,2}$ as in Corollary \ref{coeffind}. Then $b_{\lambda,h}^{(2)}=0$ for all $\ell(\lambda)+2<h \leq k$, and in particular $b_{\lambda,k}^{(2)}=0$ wherever $\ell(\lambda)<k-2$. \label{bound2}
\end{prop}

\begin{proof}
Recall from Lemma 1.3 of \cite{[CZ]} that given a polynomial $p$ of degree $d$, expanded as $\sum_{i=0}^{d}u_{i}\binom{x}{i}$, and some bound $0 \leq t \leq d$, we have $u_{i}=0$ for all $0 \leq i \leq t$ if and only if $p(m)=0$ for each $0 \leq m \leq t$. It follows that for any $s$, in the expansion $\sum_{j=0}^{d}v_{j}\binom{x-s}{j}$, we have $v_{j}=0$ for all $0 \leq j \leq t$ if and only if $p(m)=0$ for every $s \leq m \leq s+t$.

Now, in case $\ell(\lambda) \geq k-2$ the assertion holds in an empty manner (there is no $h$ in the required range), so assume that $\ell(\lambda)<k-2$. the polynomial $\sum_{h=0}^{\ell(\lambda)}(-1)^{h}a_{\lambda,h}\binom{x}{k-h}$, in which substituting $x=n$ yields $f^{(n-k,\lambda)}$ for large enough $n$ by Theorem \ref{dimform}, vanishes if we substitute $0 \leq m \leq k-\ell(\lambda)-1$ by that lemma, and this range contains at least the number $m=2$. Dividing by $\frac{x(x-1)}{2}$, representing the denominator $\binom{n}{2}$, this quotient remains a polynomial that vanishes for all $2 \leq m \leq k-\ell(\lambda)-1$.

Turning to the parentheses in Proposition \ref{Frobenius} for $(n-k,\lambda)$, one easily verifies that the only dependence on $n$ is the term $\binom{n-k}{2}$ from the first row (as the $n-k-\lambda_{1}$ columns of length 1 contribute copies of $\binom{1}{2}$ to the negative part, which is 0). Hence the parentheses form a quadratic polynomial in $n$, and in the total expression, the vanishing of the substitution of $m$ in that range remains. By (the translated) Lemma 1.3 of \cite{[CZ]}, this implies that in the expansion using $\binom{x-2}{k-h}$ of the desired polynomial, only terms in which $k-h \geq k-\ell(\lambda)-2$, or equivalently $h\leq\ell(\lambda)+2$, may appear as desired. This proves the proposition.
\end{proof}

Proposition \ref{bound2} thus translates to the equality
\begin{equation}
\chi^{(n-k,\lambda)}(\tau)=\sum_{h=0}^{\min\{k,\ell(\lambda)+2\}}(-1)^{h}b_{\lambda,h}^{(2)}\binom{n-2}{k-h}, \label{expinpols}
\end{equation}
where the superscript in the coefficient stands, when compared with the general notation below, for the fact that a transposition is a cycle of order 2.

There are only four types of partitions that do not satisfy $\ell(\lambda)<k-2$, and for which the last coefficient (with $h=k$) is not determined to vanish via Proposition \ref{bound2} and Equation (\ref{expinpols}). As in Lemma 2.6 of \cite{[CZ]}, we have to work out the details on the base cases directly.
\begin{lem}
For every $n \geq k+2$ the following equalities hold for a transposition $\tau$:
\begin{enumerate}
\item $\chi^{(n-k,1^{k})}(\tau)=\binom{n-2}{k}-\binom{n-2}{k-1}$ for $k\geq1$, and it is just $\binom{n-2}{0}=1$ when $k=0$.
\item $\chi^{(n-k,2,1^{k-2})}(\tau)=(k-1)\binom{n-2}{k}-(k-1)\binom{n-2}{k-1}+\binom{n-2}{k-2}$ wherever $k\geq2$.
\item $\chi^{(n-k,3,1^{k-3})}(\tau)=\binom{k-1}{2}\binom{n-2}{k}-\binom{k-1}{2}\binom{n-2}{k-1}+(k-2)\binom{n-2}{k-2}+\sum_{h=3}^{k}(-1)^{h} \binom{n-2}{k-h}$ if $k\geq3$
\item $\chi^{(n-k,2,2,1^{k-4})}(\tau)=\frac{k(k-3)}{2}\binom{n-2}{k}-\frac{k(k-3)}{2}\binom{n-2}{k-1}+(k-3)\binom{n-2}{k-2}-0\binom{n-2}{k-3}- \sum_{h=4}^{k}(-1)^{h}\binom{n-2}{k-h}$ for any $k\geq4$.
\end{enumerate} \label{basis2}
\end{lem}

\begin{proof}
We apply Proposition \ref{Frobenius} for $\mu=(n-k,\lambda)$ once again, where we now evaluate the dimension $f^{\lambda}$ via Theorem \ref{Hook} directly (the expression from \cite{[CZ]} is the longest for these partitions, and is thus less useful here). The term in the parentheses there is $\binom{n-k}{2}+\sum_{i}\binom{\lambda_{i}}{2}-\sum_{i}\binom{\lambda^{t}_{i}+1}{2}$.

In case 1, when $k\geq1$ the dimension is $\frac{n!}{k!n\cdot(n-1-k)!}=\binom{n-1}{k}$, and after dividing it by $\binom{n}{2}$ we get $\frac{2}{kn}\binom{n-2}{k-1}$ by expanding both binomials. As the expression in the parentheses becomes \[\textstyle{\binom{n-k}{2}-\binom{k+1}{2}=\frac{(n-k)(n-k-1)-k(k+1)}{2}=\frac{n(n-2k-1)}{2}},\] the total expression is $\frac{n-2k-1}{k}\binom{n-2}{k-1}=\frac{n-k-1}{k}\binom{n-2}{k-1}-\binom{n-2}{k-1}=\binom{n-2}{k}-\binom{n-2}{k-1}$, as desired. If $k=0$ then the representation is the trivial one, so that the value is indeed $\binom{n-2}{0}=1$.

For case 2, with $k\geq2$, the Hook Formula expresses the dimension as $\frac{n!}{(n-k-2)!(k-2)!(n-k)k(n-1)}$, whose quotient over $\binom{n}{2}$ is $\frac{2\cdot(n-2)!}{(n-k-2)!(k-2)!(n-k)k(n-1)}$, and the combination in the parentheses is \[\textstyle{\binom{n-k}{2}+1-\binom{k}{2}-1=\frac{(n-k)(n-k-1)-k(k-1)}{2}=\frac{n^{2}-(2k+1)n+2k}{2}=\frac{(n-1)(n-2k)}{2}}.\] By writing $n-2k$ as $(n-k)-k$, the first summand gives $\frac{(n-2)!}{(n-k-2)!(k-2)!k}=(k-1)\binom{n-2}{k}$. The second one becomes \[\textstyle{-\frac{(n-2)!}{(n-k-2)!(k-2)!(n-k)}=-(n-k-1)\binom{n-2}{k-2}=\binom{n-2}{k-2}-(n-k)\binom{n-2}{k-2}=\binom{n-2}{k-2}- (k-1)\binom{n-2}{k-1}},\] and the result follows.

Turning to case 3, the dimension is now $\frac{n!}{2(n-k-3)!(k-3)!(n-2)(n-k-1)(n-k)k}$, dividing which by $\binom{n}{2}$ gives $\frac{(n-2)!}{(n-k-3)!(k-3)!(n-2)(n-k-1)(n-k)k}$. The terms in the parentheses here are \[\textstyle{\binom{n-k}{2}+\binom{3}{2}-\binom{k-1}{2}-1-1=\frac{(n-k)(n-k-1)-(k-1)(k-2)}{2}+1=\frac{(n-2)(n-2k+1)}{2}+1}.\] The product with first multiplier is $\frac{(n-2k+1)(k-1)(k-2)(n-k-2)}{2}\cdot\frac{(n-2)!}{k!(n-k)!}$, which we write as $\binom{k-1}{2}$ times $\big[(n-k)(n-k-1)-k(n-k)+2k-2\big]\frac{(n-2)!}{k!(n-k)!}$, the first three of which producing the first three desired terms. The last term here and the remaining one from the $+1$ above combine to give $(k-1)(k-2)\big[-\frac{(n-2)!}{k!(n-k)!}+\frac{(n-3)!(n-2-k)}{k!(n-k)!}\big]$, where the expression in parentheses is $-\frac{(n-3)!(n-2-k)}{(k-1)!(n-k)!}$ (because $(n-2)-(n-2-k)=k$), and multiplication by $(k-1)(k-2)$ reduces to $-\binom{n-3}{k-3}$. But as the latter expression is $+\sum_{h=3}^{k}(-1)^{h}\binom{n-3}{k-h}$ as in the proof of Lemma 2.6 of \cite{[CZ]} (and noting that we start with the odd value $h=3$), this establishes the result in this case as well.

Finally, in case 4 we have the dimension $\frac{n!}{2(n-k-2)!(k-4)!(n-k+1)(n-2)(k-1)(k-2)}$, dividing by $\binom{n}{2}$ produces $\frac{(n-2)!}{(n-k-2)!(k-4)!(n-k+1)(n-2)(k-1)(k-2)}$, and the parentheses are now \[\textstyle{\binom{n-k}{2}+1+1-\binom{k-1}{2}-\binom{3}{2}=\frac{(n-k)(n-k-1)-(k-1)(k-2)}{2}-1=\frac{(n-2)(n-2k+1)}{2}-1}.\] The first multiplier now gives $\frac{k(k-3)}{2}$ times \[\textstyle{\binom{n-2}{k}\frac{n-2k+1}{n-k+1}=\binom{n-2}{k}-\frac{k}{n-k+1}\binom{n-2}{k}=\binom{n-2}{k}-\frac{n-k-1}{n-k+1}\binom{n-2}{k-1}= \binom{n-2}{k}-\binom{n-2}{k-1}+\frac{2}{n-k+1}\binom{n-2}{k-1}},\] the first of those are in the desired formula. The last term here and the remaining one combine to $\frac{k-3}{n-k+1}$ times \[\textstyle{k\binom{n-2}{k-1}-\binom{n-3}{k-1}=(k-1)\binom{n-2}{k-1}+\binom{n-3}{k-2}=(n-k)\binom{n-2}{k-2}+\binom{n-3}{k-2}= (n-k+1)\binom{n-2}{k-2}-\binom{n-3}{k-3}}\] via two applications of Pascal's identity. Recalling the external coefficient, the first summand here is the desired asserted term, and after expanding $-\frac{k-3}{n-k+1}\binom{n-3}{k-3}=-\binom{n-3}{k-4}=-\sum_{h=4}^{k}(-1)^{h}\binom{n-2}{k-h}$ (as here we start with $h=4$, which is even), the desired assertion, with the skipping over $\binom{n-2}{k-3}$, are established. This completes the proof of the lemma.
\end{proof}
In particular, Lemma \ref{basis2} produces the polynomial formula from Equation (\ref{expinpols}) for these partitions once again. Note that unlike the coefficients $a_{\lambda,h}$ from \cite{[CZ]}, which are always positive as long as $0 \leq h\leq\ell(\lambda)$, the values of the coefficients $b_{\lambda,h}^{(2)}$ from Equation (\ref{expinpols}) that show up via that lemma are not---indeed, in case 4 there it gives $b_{(2,2,1^{k-4}),h}^{(2)}=-1$ for any $4 \leq h \leq k$, while we also have the vanishing value of $b_{(2,2,1^{k-4}),3}^{(2)}$. One can verify that the resulting coefficients with $h<k$ also satisfy the equality from Corollary \ref{coeffind}.

\smallskip

In order to give the combinatorial value of the coefficients from Equation (\ref{expinpols}), we make the following definition, for which we recall the notation from Definition \ref{skewtabs}.
\begin{defn}
Fix $\lambda \vdash k$ and $0 \leq h \leq k$. If $h\geq4$ then we set $b_{\lambda,h}^{+}$ to be $f^{\lambda\setminus(3,1^{h-3})}$, and $b_{\lambda,h}^{-}:=f^{\lambda\setminus(2,2,1^{h-4})}$. When $0 \leq h\leq3$ we define $b_{\lambda,h}^{+}:=f^{\lambda\setminus(h)}$, and we set $b_{\lambda,h}^{-}$ to simply be 0. \label{vals2}
\end{defn}

We can now state and prove our main theorem for transpositions.
\begin{thm}
Let $\lambda \vdash k$ be a partition, take $n \geq k+\lambda_{1}$, and consider a transposition $\tau \in S_{n}$. Expand $\chi^{(n-k,\lambda)}(\tau)$ as in Equation (\ref{expinpols}), and then for any $0 \leq h\leq\ell(\lambda)+2$, the coefficient $b_{\lambda,h}^{(2)}$ from that equation equals the difference $b_{\lambda,h}^{+}-b_{\lambda,h}^{-}$ of the numbers from Definition \ref{vals2}. \label{transthm}
\end{thm}

\begin{proof}
Take any $h\geq0$, and $\lambda \vdash k$. If $k<h$ then $b_{\lambda,h}^{(2)}=0$ in Proposition \ref{bound2} and Equation (\ref{expinpols}), and since $\lambda$ cannot contain a partition $\nu \vdash h$, we get that $b_{\lambda,h}^{+}=b_{\lambda,h}^{-}=0$ as well via Definitions \ref{vals2} and \ref{skewtabs}. The same argument shows that $b_{\lambda,h}^{(2)}=b_{\lambda,h}^{+}=b_{\lambda,h}^{-}=0$ in case $k=h$ and $\ell(\lambda)<k-2$.

When $k=h$ and $\ell(\lambda) \geq k-2$, then $\lambda$ is one of the partitions from Lemma \ref{basis2}, and we need to compare the constant term of the expansion there with the asserted value. Note that since $k=h$, the only situation in which the partition of $h$ from Definition \ref{vals2} is contained in $\lambda$ is where it equals $\lambda$.

For case 1 this happens for $k=0$ and $k=1$ but not otherwise. In case 2 this occurs only when $k=2$. Considering case 3, we have the equality for the partition associated with $b_{\lambda,h}^{+}$ for every value of $k\geq3$. When the partition is the one from case 4, then it corresponds to the one yielding $b_{\lambda,h}^{-}$. Comparing these with the constant terms from Lemma \ref{basis2}, and recalling the sign $(-1)^{h}=(-1)^{k}$ and the negative sign in front of $b_{\lambda,h}^{-}$, we obtain the result for $k=h$.

For $k>h$ we argue by induction, assuming that the result holds for all partitions of $k-1$. Since $b_{\lambda,h}^{(2)}=c_{\lambda,h}^{\tau,2}$ and $h<k$, Corollary \ref{coeffind} (with $\sigma=\tau$ and $s=2$) expresses it as $\sum_{v\in\operatorname{IC}(\lambda)}b_{\lambda-v,h}^{(2)}$. But as $\lambda-v \vdash k-1$ for every such $v$, the induction hypothesis expresses the corresponding summand as $b_{\lambda-v,h}^{+}-b_{\lambda-v,h}^{-}$.

But as Definition \ref{vals2} expresses each of the latter summands as $f^{(\lambda-v)\setminus\nu}$ for $\nu \vdash h$ (or as 0), and $k>h$ so $\lambda\neq\nu$, we can apply Lemma \ref{recskew}, and obtain that $\sum_{v\in\operatorname{IC}(\lambda)}b_{\lambda-v,h}^{\pm}=b_{\lambda,h}^{\pm}$ for each of the signs (including for the negative one when $0 \leq h\leq3$, where both sides vanish). In total, we obtain the expression for $\lambda$ and $h$ as desired. This completes the proof of the theorem.
\end{proof}
In fact, the result of Theorem \ref{transthm} also holds for $\ell(\lambda)+2<h \leq k$, as both sides vanish for such values of $h$ (one via Proposition \ref{bound2} and Equation (\ref{expinpols}), and the other since the two partitions of $h$ from Definition \ref{vals2}, whose length is at least $h-2$, cannot fit inside $\lambda$ of a shorter length). That theorem also yields $b_{\lambda,h}^{(2)}=a_{\lambda^{t},h}$ for $0 \leq h\leq3$ using the notation of Theorem \ref{dimform}, so that in particular $b_{\lambda,0}^{(2)}=b_{\lambda,1}^{(2)}=f^{\lambda}$ (compare the value for $h=0$ with Corollary \ref{specs} below).

\section{The Result for Arbitrary Cycles \label{GenCyc}}

Fix now $r\geq2$, and let $\sigma \in S_{n}$ be a cycle of length $r$. When we apply the Murnaghan--Nakayama Rule from Theorem \ref{MNrule} to the partition $\mu=(n-k,\lambda) \vdash n$ for some $\lambda \vdash k$, where $n \geq k+\lambda_{1}+r$, we get that if $\rho\in\operatorname{SH}_{r}^{(n-k,\lambda)}$ touches the first row then it is contained in it, and otherwise it is, up to translations, and element of $\operatorname{SH}_{r}^{\lambda}$. This expresses $\chi^{(n-k,\lambda)}(\sigma)$ as $f^{(n-k-r,\lambda)}+\sum_{\rho\in\operatorname{SH}_{r}^{\lambda}}(-1)^{\mathrm{ll}(\rho)}f^{((n-r)-(k-r),\lambda\setminus\rho)}$, where we wrote $n-k$ as $(n-r)-(k-r)$ since $\lambda\setminus\rho \vdash k-r$ for $\rho\in\operatorname{SH}_{r}^{\lambda}$, and as Theorem \ref{dimform} expresses those terms using the binomial coefficients $\big\{\binom{n-r}{m}\big\}_{m\geq0}$, we take $s=r$ in this setting.

We thus adopt the notation similar to that from Proposition \ref{bound2} and Equation (\ref{expinpols}), and write
\begin{equation}
\chi^{(n-k,\lambda)}(\sigma)=\sum_{h=0}^{k}(-1)^{h}b_{\lambda,h}^{(r)}\binom{n-r}{k-h}, \label{expgen}
\end{equation}
where $b_{\lambda,h}^{(r)}$ is the coefficient $c_{\lambda,h}^{\sigma,r}$ for our cycle $\sigma$ of length $r$ in the notation from Corollary \ref{coeffind}. We recall again that this corollary determines all the coefficients with index $0 \leq h<k$ in Equation (\ref{expgen}) by the values associated with smaller partitions, so that we only need to find the constant coefficient, associated with $h=k$.

\smallskip

Recall from Definition \ref{vals2} and Theorem \ref{transthm} that for transpositions, the $h$th coefficient in the polynomial from Equation (\ref{expinpols}) that is associated with $\lambda \vdash k$ involved two numbers, each of which is of the form $f^{\lambda\setminus\nu}$ for some $\nu \vdash h$ as in Definition \ref{skewtabs} (or 0), with opposite signs. These partitions $\nu \vdash h$ were precisely those for which the last coefficient in Lemma \ref{basis2} (with $k=h$) did not vanish, and the sign arose from the one appearing in the constant term in that lemma.

For our cycle $\sigma$, of general length $r\geq2$, we thus make the following definition.
\begin{defn}
We call the following partitions \emph{$r$-primary}, and to each of them we attach the corresponding \emph{$r$-sign}:
\begin{enumerate}
\item The partitions $(1^{t}) \vdash t$ for $0 \leq t \leq r-1$, with the $r$-sign being positive.
\item Given $u,v\geq0$ with $u+v \leq r-2$, we take $(r-u-v,2^{u},1^{v}) \vdash r+u$, with $(-1)^{r-u-v}$.
\item Take $0 \leq u \leq r-1$ and $v\geq0$, and to $(r+1-u,2^{u},1^{v}) \vdash r+1+u+v$ we attach $(-1)^{r-u}$.
\end{enumerate}
For any $h\geq0$ we denote by $\Gamma^{r}_{h}$ the set of partitions $\nu \vdash h$ that are $r$-primary, and given $\nu\in\Gamma^{r}_{h}$, we denote by $\varepsilon^{r}_{\nu}$ the corresponding sign. \label{rprim}
\end{defn}

\begin{ex}
Here are the 3-primary partitions with their 3-signs and sizes:
\[\begin{tabular}{|c||c|c|c||c|c|c||c|c|c|} \hline $\nu$ & $\emptyset$ & (1) & (1,1) & (3) & (2,1) & (2,2) & (4,$1^{v}$) & (3,2,$1^{v}$) & (2,2,2,$1^{v}$) \\ \hline $\varepsilon^{3}_{\nu}$ & + & + & + & $-$ & + & + & $-$ & + & $-$ \\ \hline $h$ & 0 & 1 & 2 & 3 & 3 & 4 & $4+v$ & $5+v$ & $6+v$ \\ \hline \end{tabular}\] \label{r3ex}
\end{ex}
The double lines in Example \ref{r3ex} separate between the three different types from Definition \ref{rprim}. The value of $v$ in the third type is arbitrary. We see that the size of $\Gamma_{3}^{h}$ is 1 when $h$ is 0, 1 or 2, it equals 2 in case $h$ equals 3, 4, or 5 (with the small values of $v$ in the third type), and it is 3 for any $h\geq6$.

\begin{rmk}
In general, it is clear from Definition \ref{rprim} that if $0 \leq h<r$ then $\Gamma^{r}_{h}=\{(1^{h})\}$, hence of size 1. For $h\geq2r$ it contains only partitions of the third type, with all values of $0 \leq u \leq r-1$ and $v=h-1-r-u$, hence $r$ partitions in total. In case $r \leq h<2r$ there are the possible values $0 \leq u \leq h-1-r$ and $v=h-1-r-u$ in the third type (which is empty in case $h=r$), as well as those of the second type in which $u=h-r$ and $0 \leq v \leq 2r-2-h$ (which exist only for $h<2r-1$), and in total this gives $r-1$ partitions. \label{sizeGamma}
\end{rmk}

We can now determine the last coefficient, with $h=k$, in Equation (\ref{expgen}).
\begin{lem}
The coefficient $b_{\lambda,k}^{(r)}$, when $\lambda \vdash k$ and $h=k$, equals the $r$-sign $\varepsilon^{r}_{\lambda}$ in case $\lambda$ is $r$-primitive, and vanishes otherwise. \label{genbasis}
\end{lem}

\begin{proof}
We recall that $\chi^{(n-k,\lambda)}(\sigma)$ equals $f^{(n-r-k,\lambda)}+\sum_{\rho\in\operatorname{SH}_{r}^{\lambda}}(-1)^{\mathrm{ll}(\rho)}f^{((n-r)-(k-r),\lambda\setminus\rho)}$ by Theorem \ref{MNrule}, and we express each term using Theorem \ref{dimform}. After replacing $h$ by $h+r$ in the sum associated with every $\rho\in\operatorname{SH}_{r}^{\lambda}$, we get the expression \[\sum_{h\geq0}(-1)^{h}\bigg[a_{\lambda,h}+\sum_{\rho\in\operatorname{SH}_{r}^{\lambda}}(-1)^{r-\mathrm{ll}(\rho)}a_{\lambda\setminus\rho,h-r}\bigg]\binom{n-r}{k-h},\] where the latter summands show up only for $h \geq r$. This yields, in particular, another proof of Proposition \ref{polgen} for this case, and $b_{\lambda,h}^{(r)}$ is the expression inside the parentheses.

We thus need to check when do contributions to the case $k=h$ exist. As the sum in Theorem \ref{MNrule} goes up to $\ell(\lambda)$, the first summand is 1 when $\lambda=(1^{k})$ and vanishes otherwise. Similarly, as $\lambda\setminus\rho \vdash k-r$, the contribution for $h=k$ from the associated with $\rho$ is $(-1)^{r-\mathrm{ll}(\rho)}$ in case $\lambda\setminus\rho=(1^{k-r})$, and is 0 in any other case. Therefore the only possible partitions $\lambda \vdash k$ for which $b_{\lambda,k}^{(r)}\neq0$ are $(1^{k})$ and those obtained by adding a skew-hook of length $r$ to $(1^{k-r})$.

Now, if $\lambda=(1^{k})$ and $k<r$ then $\lambda\in\Gamma^{r}_{k}$ by the first case of Definition \ref{rprim}. When $k \geq r$, we also have the element $\rho\in\operatorname{SH}_{r}^{\lambda}$ consisting of the boxes in the last $r$ rows, for which we have $\mathrm{ll}(\rho)=r-1$ and thus its contribution cancels with the one from the first term (and this is clearly the only element of $\operatorname{SH}_{r}^{\lambda}$). It therefore remains to show that given a partition $(1^{k})\neq\lambda \vdash k$, there exists an element $\rho\in\operatorname{SH}_{r}^{\lambda}$ for which $\lambda\setminus\rho=(1^{k-r})$ if and only if $\lambda\in\Gamma^{r}_{k}$ via one of the other two types, and in that case $(-1)^{r-\mathrm{ll}(\rho)}=\varepsilon^{r}_{\lambda}$ (since $\rho$ is then clearly the only element of $\operatorname{SH}_{r}^{\lambda}$ for which $\lambda\setminus\rho=(1^{k-r})$ and we get a contribution to $b_{\lambda,k}^{(r)}$).

Now, the fact that $\lambda\neq(1^{k})$ implies that $(1,2)\in\lambda$, and as $\lambda\setminus\rho=(1^{k-r})$ we deduce that $(1,2)$ must be in $\rho$ as well. But as $\rho$ is a skew-hook in $\lambda$, we deduce that $(2,3)\not\in\lambda$, and that $\rho$ contains the rest of the first row of $\lambda$ except for $(1,1)$, and all the boxes in the second column of $\lambda$. The fact that $(2,3)\not\in\lambda$ implies that $\lambda_{2}\leq2$, so that if we denote by $u\geq0$ the number of rows in $\lambda$, below the first one, that are of length 2, and by $v\geq0$ the number of those that are of length 1, then $\lambda=(\lambda_{1},2^{u},1^{v})$ (in fact, in this case $2u+v \leq k-2$ and $\lambda_{1}=k-2u-v\geq2$, but we shall not use this value).

We now distinguish among two situations:
\begin{enumerate}
\item $\rho$ passes through the first column of $\lambda$ as well. Then the picture looks like 
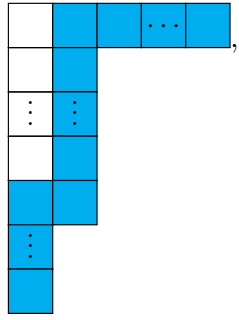
\[\begin{ytableau} ~ & *(cyan) & *(cyan) & *(cyan) \cdots & *(cyan) \\ ~ & *(cyan) \\ \raisebox{-0.3ex}{\vdots} & *(cyan) \raisebox{-0.3ex}{\vdots} \\ ~ & *(cyan) \\ *(cyan) & *(cyan) \\ *(cyan) \raisebox{-0.3ex}{\vdots} \\ *(cyan) \\ \end{ytableau},\] with the white boxes representing $\lambda\setminus\rho=(1^{k-r})$. Then $\rho$ contains $v+u+1$ boxes away from the first row and at least the box $(1,2)$, and as it is of size $r$, it contains $r-v-u-1\geq1$ boxes in the first row and we get $u+v \leq r-2$. Therefore $\lambda=(r-u-v,2^{u},1^{v})\in\Gamma^{r}_{k}$ via the second case of Definition \ref{rprim}, and as $\mathrm{ll}(\rho)=u+v$ we deduce that $(-1)^{r-\mathrm{ll}(\rho)}=(-1)^{r-u-v}=\varepsilon^{r}_{\lambda}$ as desired.
\item $\rho$ does not pass through the first column of $\lambda$. The picture here is 
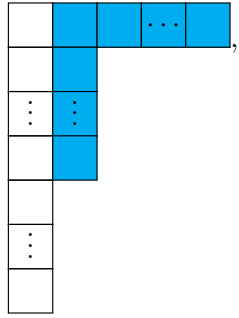
\[\begin{ytableau} ~ & *(cyan)& *(cyan) & *(cyan)\cdots & *(cyan) \\ ~ & *(cyan) \\ \raisebox{-0.3ex}{\vdots} & *(cyan) \raisebox{-0.3ex}{\vdots} \\ ~ & *(cyan) \\ ~ \\ \raisebox{-0.3ex}{\vdots} \\ ~ \\ \end{ytableau},\] again with $\lambda\setminus\rho=(1^{k-r})$ being described by the set of white boxes. Then it is clear that $\rho$ contains $r-u\geq1$ boxes in the first row, and thus $\lambda=(r+1-u,2^{u},1^{v})$ with every $0 \leq u \leq r-1$ and $v\geq0$ being possible, so that it is in $\Gamma^{r}_{k}$ by the third case of Definition \ref{rprim}. The fact that $\mathrm{ll}(\rho)=u$ here implies that $(-1)^{r-\mathrm{ll}(\rho)}=(-1)^{r-u}=\varepsilon^{r}_{\lambda}$ as asserted.
\end{enumerate}
In total, when $\lambda\in\Gamma^{r}_{k}$ we get $b_{\lambda,k}^{(r)}=\varepsilon^{r}_{\lambda}$, and otherwise $b_{\lambda,k}^{(r)}=0$. This proves the lemma.
\end{proof}
One easily sees that the sum over $h$ in the proof of Lemma \ref{genbasis} goes up to $\min\{k,\ell(\lambda)+r\}$, and indeed all the $r$-primary partitions from Definition \ref{rprim} have length larger than $k-r$. But as in general some partitions can have such longer length without being $r$-primary (like $(3,3,1^{h-6})$ for $r=4$), we need the precise set from the latter definition.

\smallskip

We may now state our result for general cycles.
\begin{thm}
Let $\lambda \vdash k$ be a partition, fix $r\geq2$, take $n\geq\max\{k+\lambda_{1},r\}$, and let $\sigma \in S_{n}$ be an $r$-cycle. Then when we express the value $\chi^{(n-k,\lambda)}(\sigma)$ of the character of the representation associated with the representation corresponding to the partition $(n-k,\lambda) \vdash n$ at $\sigma$ via Equation (\ref{expgen}), the value of the coefficient $b_{\lambda,h}^{(r)}$ for any $0 \leq h\leq\min\{k,\ell(\lambda)+r\}$ is given by the combination $\sum_{\nu\in\Gamma^{r}_{h}}\varepsilon^{r}_{\nu}f^{\lambda\setminus\nu}$ of the expressions from Definition \ref{skewtabs} multiplied by the signs from Definition \ref{rprim}. \label{main}
\end{thm}

\begin{proof}
As in the proof of Theorem \ref{transthm}, we fix $h\geq0$ and argue by induction on $k$, for all partitions $\lambda \vdash k$. When $k<h$ we have $b_{\lambda,h}^{(r)}=0$ in Equation (\ref{expgen}) for every $\lambda \vdash k$, and as such $\lambda$ cannot contain any $\nu \vdash h$ when $k<h$, we have $f^{\lambda\setminus\nu}=0$ for every $\nu\in\Gamma^{r}_{h}$.

Assume now that $k=h$, where it is clear that if $\nu\subseteq\lambda$ for $\nu \vdash h$ then $\lambda=\nu$. The expression $\sum_{\nu\in\Gamma^{r}_{h}}\varepsilon^{r}_{\nu}f^{\lambda\setminus\nu}$ therefore reduces to $\varepsilon^{r}_{\lambda}$ when $\lambda\in\Gamma^{r}_{h}$, and to 0 otherwise. As Lemma \ref{genbasis} shows that $b_{\lambda,h}^{(r)}=b_{\lambda,k}^{(r)}$ attains the same value, the result holds in this case as well.

We therefore assume that $k>h$, and that the result holds for every partition of $k-1$. We can thus apply Corollary \ref{coeffind} (with our $\sigma$ and with $s=r$), followed by the induction hypothesis applied to $\lambda-v \vdash k-1$ for every $v\in\operatorname{IC}(\lambda)$, and then, after interchanging the summation order, the fact that $\lambda \vdash k$ cannot equal any $\nu\in\Gamma^{r}_{h}$ allows us to invoke Lemma \ref{recskew}. In total we get \[b_{\lambda-v,h}^{(r)}=\sum_{v\in\operatorname{IC}(\lambda)}b_{\lambda-v,h}^{(r)}=\sum_{v\in\operatorname{IC}(\lambda)}\sum_{\nu\in\Gamma^{r}_{h}}\varepsilon^{r}_{\nu} f^{(\lambda-v)\setminus\nu}=\sum_{\nu\in\Gamma^{r}_{h}}\varepsilon^{r}_{\nu}\sum_{v\in\operatorname{IC}(\lambda)}f^{(\lambda-v)\setminus\nu}= \sum_{\nu\in\Gamma^{r}_{h}}\varepsilon^{r}_{\nu}f^{\lambda\setminus\nu},\] and the result holds for our $\lambda$ as well. This completes the proof of the theorem.
\end{proof}
Note that for $r=2$ the 2-primary partitions are $\emptyset$ and $(1)$ with a positive sign arising from the first type of Definition \ref{rprim}, $(2)$ from the second type there (with $u=v=0$) also with a positive sign, and the third type produces $(3,1^{v})$ with $u=0$ and $(2,2,1^{v})$ with $u=1$, with a positive and negative sign respectively. Hence if $0 \leq h\leq3$ then $\Gamma^{2}_{h}$ contains the single partition $(h)$ with a positive sign, and when $h\geq4$ this set has two elements (as predicted by Remark \ref{sizeGamma}). Comparing these signs with Definition \ref{vals2} shows that Theorem \ref{main} for $r=2$ yields Theorem \ref{transthm}.

\begin{rmk}
In fact, we did not use the assumption that $r\geq2$ at all in our arguments. When we take $r=1$ in Definition \ref{rprim}, we get $\Gamma^{1}_{0}=\{\emptyset\}$ and $\varepsilon^{1}_{\emptyset}=+$ (like for any $r$), $\Gamma^{1}_{1}=\emptyset$ (since there are no $u\geq0$ and $v\geq0$ summing to $r-2=-1$), and $\Gamma^{1}_{h}=\{(2,1^{h-2}\}$ for any $h\geq2$ via the third case there (with $u=0$), yielding the sign $-1$. Theorem \ref{main} thus implies that $\chi^{(n-k,\lambda)}(\sigma)=f^{(n-k,\lambda)}$ is given by $f^{\lambda}\binom{n-1}{k}-\sum_{h=2}^{k}(-1)^{h}f^{\lambda\setminus(2,1^{h-2})}\binom{n-1}{k-h}$. \label{r1alt}
\end{rmk}
Remark \ref{r1alt} generalizes the fact that $f^{(n-k,1^{k})}=\binom{n-1}{k}$, since $(1^{k})$ does not contain $(2,1^{h-2})$ for any $h\geq2$. It differs from Theorem \ref{dimform} since one uses $s=0$ in Corollary \ref{coeffind}, and the other is with $s=1$. In fact, we obtain the following consequence.
\begin{cor}
For every $\lambda \vdash k$ and every $2 \leq h\leq\ell(\lambda)$ we have $f^{\lambda\setminus(1^{h-1})}-f^{\lambda\setminus(1^{h})}=f^{\lambda\setminus(2,1^{h-2})}$. \label{relr1}
\end{cor}

\begin{proof}
Write $f^{(n-k,\lambda)}$ as $\sum_{h=0}^{k}(-1)^{h}f^{\lambda\setminus(1^{h})}\binom{n}{k-h}$ via Theorem \ref{dimform} and the interpretation of the coefficients there via Definition \ref{skewtabs}. We can now write $\binom{n}{k-h}$ as $\binom{n-1}{k-h}+\binom{n-1}{k-1-h}$ (also for $h=k$ when $\lambda=(1^{k})$, with the second term vanishing there), and after replacing $h$ by $h-1$ in the second term of each summand, the result is $f^{\lambda}\binom{n-1}{k}+\sum_{h=1}^{k}(-1)^{h}[f^{\lambda\setminus(1^{h})}-f^{\lambda\setminus(1^{h-1})}]\binom{n-1}{k-h}$. Comparing the terms with $h\geq2$ with those appearing in Remark \ref{r1alt} and noting the opposite sign yields the desired result. This proves the corollary.
\end{proof}
In the extreme case $\lambda=(1^{k})$, Corollary \ref{relr1} reduces to the equality $1-1=0$. The vanishing of the term associated with $h=1$ in Remark \ref{r1alt} corresponds, via the proof of that corollary, to the equality $a_{\lambda,1}=a_{\lambda,0}=f^{\lambda}$ from Proposition 3.1 of \cite{[CZ]} (or the vanishing in Theorem D of \cite{[Ra]}). In terms of $f^{\lambda\setminus(1^{h-1})}$ as counting elements of $\operatorname{SYT}(\lambda)$ where the first $h-1$ numbers are at the beginning of the first $h-1$ rows, Corollary \ref{relr1} is visible in the fact that $h$ can be located either at the beginning of the $h$th row (yielding an element of $f^{\lambda\setminus(1^{h})}$), or after 1 in the first row (producing an element of $f^{\lambda\setminus(2,1^{h-2})}$).

\smallskip

We now fix $\lambda \vdash k$, and consider the behavior of the expression from Theorem \ref{main} as $r$ varies. We therefore denote by $\sigma_{r}$ a cycle of length $r$ for any $r\geq2$ (so that $\tau=\sigma_{2}$).
\begin{ex}
Consider the partition $\lambda=(3,3)$. Theorems \ref{dimform} and \ref{main} then give
\begin{alignat*}{6}
\chi^{(n-6,3,3)}(\tau) & =5\binom{n-2}{6} & -5\binom{n-2}{5} & +3\binom{n-2}{4} & -1\binom{n-2}{3} &  & \\
\chi^{(n-6,3,3)}(\sigma_{3}) & =5\binom{n-3}{6} & -5\binom{n-3}{5} & +2\binom{n-3}{4} & -1\binom{n-3}{3} & +1\binom{n-3}{2} & -1\binom{n-3}{1} \\
\chi^{(n-6,3,3)}(\sigma_{4}) & =5\binom{n-4}{6} & -5\binom{n-4}{5} & +2\binom{n-4}{4} &  & -1\binom{n-4}{2} & +1\binom{n-4}{1} \\
\chi^{(n-6,3,3)}(\sigma_{r\geq5}) &=5\binom{n-r}{6} & -5\binom{n-r}{5} & +2\binom{n-r}{4} & & & \\
f^{(n-6,3,3)} &= 5\binom{n}{6} & -5\binom{n}{5} \qquad & +2\binom{n}{4}. \qquad & & &
\end{alignat*} \label{difrex}
\end{ex}
We can write $f^{(n-6,3,3)}$ in Example \ref{difrex} also as $5\binom{n-1}{6}-3\binom{n-1}{4}+2\binom{n-1}{3}$ via Remark \ref{r1alt}.

We now consider the limit behavior as $r\to\infty$ for the coefficients from Theorem \ref{main}.
\begin{cor}
For fixed $\lambda \vdash k$ and $0 \leq h \leq k$, the coefficients $b_{\lambda,h}^{(r)}$ in the expansion of $\chi^{(n-k,\lambda)}(\sigma_{r})$ in Equation (\ref{expgen}) stabilize as $r\to\infty$, and are equal to $a_{\lambda,h}=f^{\lambda\setminus(1^{h})}$ once $r>h$. \label{limr}
\end{cor}

\begin{proof}
For $h<r$ Remark \ref{sizeGamma} shows that $\Gamma^{r}_{h}$ from Definition \ref{rprim} is the singleton $\{(1^{h})\}$, with $\varepsilon^{r}_{1^{h}}=+1$. Hence the assertion follows directly from Theorem \ref{main}. This proves the corollary.
\end{proof}
Example \ref{difrex} exemplifies the behavior described in Corollary \ref{limr}.

\section{A Second Determination of the Constant Coefficients \label{AltPf}}

The proof of Lemma \ref{genbasis}, which forms the basic case for Theorem \ref{main}, was based on the Murnaghan--Nakayama Rule, as described in Theorem \ref{MNrule}. It is based on taking one cycle out of an element $\eta \in S_{n}$, and describing the character values on $\eta$ using character values of on an element with this one cycle removed.

Our results, as are those of \cite{[CZ]}, are based on taking the longest row out of $\mu$, namely $n-k$ in case $\mu=(n-k,\lambda)$ and $n \geq k+\lambda_{1}$, and describing the corresponding character values in terms of quantities that are associated with the remaining partition $\lambda$. There is a result given in terms of this point of view, which we now describe.

Recall that if $\lambda$ contains another partition $\kappa$ then we call $\lambda\setminus\kappa$ a \emph{vertical strip} if it occupies at most one box in each row. We denote the set of vertical strips thus obtained inside $\lambda$ by $\operatorname{VS}_{\lambda}$. For some partition $\alpha \vdash n$, we denote by $a_{i}$ the multiplicity of $i$ in $\alpha$, so that $\alpha$ is the appropriate re-ordering of $(1^{a_{1}},2^{a_{2}},\ldots,n^{a_{n}})$ (with trivial powers $i^{0}$ removed, of course).

Given a permutation $\sigma \in S_{n}$, we write $x_{i}(\sigma)$ for the number of cycles in $\sigma$ that are of length $i$, so that in particular $x_{1}(\sigma)$ is the number of fixed points of $\sigma$, and we have $\sum_{i=1}^{n}ix_{i}(\sigma)=n$. Given $\alpha \vdash n$ with $\{a_{i}\}_{i=1}^{n}$ as above and variables $\{x_{i}\}_{i=1}^{n}$, we write $\binom{x}{\alpha}$ for the product $\prod_{i=1}^{n}\binom{x_{i}}{a_{i}}$, and we denote by $\pi_{\alpha}$ some permutation in $S_{n}$ with cycle structure that is described by $\alpha$.

We now cite a result from in Section 2.9 of \cite{[P]}, or Equation (5) on page 124 of 
\cite{[M]} (the latter after substituting $N=n-k$ and evaluating at an element of $S_{n}$).
\begin{prop}
Take a partition $\lambda \vdash k$ and $\sigma \in S_{n}$ for some $n \geq k+\lambda_{1}$. Then we have \[\chi^{(n-k,\lambda)}(\sigma)=\sum_{\substack{\kappa\subseteq\lambda \\ \lambda\setminus\kappa\in\operatorname{VS}_{\lambda}}}(-1)^{|\lambda\setminus\kappa|}\sum_{\alpha\vdash|\kappa|}\chi^{\kappa}(\pi_{\alpha})\binom{x(\sigma)}{\alpha}.\] \label{recpart}
\end{prop}

In fact, Proposition \ref{recpart} yields another proof of Proposition \ref{polgen} and suggests a particular value for $s$ in Corollary \ref{coeffind}.
\begin{cor}
Let $\sigma$ be fixed as in Proposition \ref{polgen}. Then $\chi^{(n-k,\lambda)}(\sigma)$ is a the value at $n$ of a polynomial of degree $k$, which is given as linear combinations of $\big\{\binom{x-s}{k-h}\big\}_{h=0}^{k}$ for $s:=\sum_{i\geq2}ix_{i}(\sigma)$, and where the coefficient multiplying $\binom{x-s}{k}$ is $f^{\lambda}>0$. \label{specs}
\end{cor}

\begin{proof}
The fact that $\sigma$ is fixed implies that so is $x_{i}(\sigma)$ for any $i\geq2$, our $s$ is a constant as well and we have $x_{1}(\sigma)=n-s$. Therefore every choice of $\kappa$ and $\alpha$ in Proposition \ref{polgen} yields some constant times $\binom{n-s}{a_{1}}$, yielding the polynomial property and the form of the expansion from Corollary \ref{coeffind}, with our $s$.

For the degree we need to maximize $a_{1}$, which for fixed $|\kappa|$ is obtained when $\alpha=1^{|\kappa|}$ and $a_{1}=|\kappa|$. As this is maximized when $\kappa=\lambda$, yielding a single term involving $\binom{n-s}{k}$, and all the others being of smaller degree. Hence the degree of our polynomial is indeed $k$, and as for $\kappa=\lambda$ and $\alpha=1^{k}$ the sign is $+$ and the element $\pi_{\alpha}$ is trivial, we get the asserted multiplier as well. This proves the corollary.
\end{proof}
In the terminology from Corollary \ref{coeffind}, Corollary \ref{specs} states that $c_{\lambda,0}^{\sigma,s}=f^{\lambda}$ for our $\sigma$ and $s$, and therefore this is the same for every choice of $s$ (as the leading coefficient of our polynomial, up to the factor $k!$). Note that when $\sigma$ is a cycle of length $r$, the value of $s$ in that proposition is indeed $r$ as in Theorem \ref{main}, so that this proposition also established Equation (\ref{expgen}), generalized to any permutation.

\smallskip

In general, Proposition \ref{recpart} does not give all the coefficients in Equation (\ref{expgen}), as even for cycles there may be several contributions to a fixed $b_{\lambda,h}^{(r)}$. However, as we saw in our proofs, once the basic case for $h$ is given in terms of a linear combination of numbers counting standard Young tableaux with fixed skew-shapes of size $h$ (as in Definition \ref{skewtabs}), it is valid for every partition using Corollary \ref{coeffind} and Lemma \ref{recskew}. We thus use this result in order to obtain the constant terms from Lemma \ref{genbasis}.

We shall be using the following consequence of the Murnaghan--Nakayama Rule.
\begin{lem}
If $\kappa \vdash r$ and $\sigma$ is a cycle of length $r$ in $S_{r}$ Then $\chi^{\kappa}(\sigma)$ equals $(-1)^{r-i}$ in case $\kappa$ is of the form $(i,1^{r-i})$ for some $1 \leq i \leq r$, and it vanishes otherwise. \label{charcyc}
\end{lem}

\begin{proof}
We apply Theorem \ref{MNrule} with $\mu=\kappa$ and $n=r$. Then $\operatorname{SH}_{r}^{\kappa}$ equals $\{\kappa\}$ in case $\kappa$ is a skew-hook (and then the multiplier involving $\chi^{\kappa\setminus\kappa}$ is 1), and is empty otherwise. As the only case where $\kappa$ is a skew-hook is when $\kappa=(i,1^{r-i})$ for such $i$, and then $\mathrm{ll}(\kappa)=r-i$ (also when $i=1$, representing $\kappa=(1^{r})$), the assertion follows. This proves the lemma.
\end{proof}

We may now give a second proof for our result.
\begin{proof}[Second proof of Lemma \ref{genbasis}]
Contributions to the constant term in Proposition \ref{recpart} and Corollary \ref{specs} arise from summands $\alpha$ for which $a_{1}=0$. In the remaining product $\prod_{i=2}^{n}\binom{x_{i}(\sigma)}{a_{i}}$ forming $\binom{x(\sigma)}{\alpha}$, we note that as $\sigma$ is a cycle of length $r$ we have $x_{r}(\sigma)=1$ and $x_{i}(\sigma)=0$ for any other $i\geq2$. Therefore only $\alpha$ for which $a_{i}=0$ for $r \neq i\geq2$ and $a_{r}\in\{0,1\}$ contribute, and we also have $a_{1}=0$. This produces only two options, namely either $\alpha$ is empty and hence so is $\kappa$, or $\kappa \vdash r$ and $\pi_{\alpha} \in S_{r}$ is a full cycle, namely coincides with our $\sigma$ in that group.

We recall the sign $(-1)^{|\lambda\setminus\kappa|}$, as well as the sign $(-1)^{h}=(-1)^{k}$ from Equation (\ref{expgen}) with $h=k$, and the fact that $\lambda \vdash k$. Hence if $\alpha=\kappa=\emptyset$ contributes a value then this value is $+\chi^{\emptyset}=1$, and every element $\kappa\subseteq\lambda$ with $|\kappa|=r$ and $\lambda\setminus\kappa\in\operatorname{VS}_{\lambda}$ contributes a value of $(-1)^{r}\chi^{\kappa}(\sigma)$.

The case where $\alpha=\kappa=\emptyset$ appears precisely when $\lambda\in\operatorname{VS}_{\lambda}$, namely if and only that $\lambda$ is a vertical strip, namely just for $\lambda=(1^{k})$. If $k<r$ then $\lambda$ cannot contain any $\kappa$ of size $r$, and thus the total value is $+1=\varepsilon^{r}_{\lambda}$ via the first case of Definition \ref{rprim}. When $k \geq r$ we can take $\kappa=(1^{r})$, where indeed $\lambda\setminus\kappa\in\operatorname{VS}_{\lambda}$, and since $(-1)^{r}\chi^{\kappa}(\sigma)=-1$ (either via Lemma \ref{charcyc} or simply by checking the sign of $\sigma$), the two summands cancel out of $b_{\lambda,k}^{(r)}=0$ as desired.

We thus assume that $\lambda\neq(1^{k})$, and Lemma \ref{charcyc} allows us to restrict attention to those $\lambda \vdash k \geq r$ that contain $\kappa=(i,1^{r-i})$ for some $1 \leq i \leq r$ such that $\lambda\setminus\kappa\in\operatorname{VS}_{\lambda}$, and each such $\kappa$ will contribute $(-1)^{r}\chi^{\kappa}(\sigma)=(-1)^{i}$ to the constant term. Moreover, since in this situation the number $\lambda_{1}-\kappa_{1}=\lambda_{1}-i$ in the first row of the vertical strip $\lambda\setminus\kappa$ can be either 0 or 1, there can be at most two such $\kappa$'s, with consecutive values of $i$. But as when both values produce vertical strips the their contributions cancel, we only get non-zero a constant term for $\lambda\neq(1^{k})$ in case precisely one value of $1 \leq i \leq r$ yields $\kappa=(i,1^{r-i})\subseteq\lambda$ and $\lambda\setminus\kappa\in\operatorname{VS}_{\lambda}$.

We note that our $\kappa$ takes at most one box from every row below the first one, and $\lambda\setminus\kappa$ is a vertical strip, so that $\lambda$ must be of the sort $(\lambda_{1},2^{u},1^{v})$ where $u \leq r-i \leq u+v$. Moveover, the vertical strip condition implies, in the first row, that $\lambda_{1}-\kappa_{1}=\lambda_{1}-i$ can be either 0 or 1, and we consider each of these options separately.
\begin{enumerate}
\item When $\lambda_{1}-\kappa_{1}=0$, we get $\lambda_{1}=i\geq2$ (because $\lambda\neq(1^{k})$), and we claim that $u+v=r-i$. Indeed, otherwise we have $\lambda_{r-i+2}=1$ (since $\kappa_{r-i+2}=0$ and $\lambda\setminus\kappa\in\operatorname{VS}_{\lambda}$), and then putting the box in that row instead of $(1,i)$ in $\kappa$ produces another element of the form from Lemma \ref{charcyc} inside $\lambda$ with complement in $\operatorname{VS}_{\lambda}$, and we assumed that this is not the case. Thus $u+v=r-i \leq r-2$, namely $\lambda\in\Gamma^{r}_{k}$ as in the second case of Definition \ref{rprim}, and the constant term is $(-1)^{i}=(-1)^{r-u-v}=\varepsilon^{r}_{\lambda}$ as desired. Here is $\lambda$, with $\kappa$ consisting of the cyan boxes in this case. 
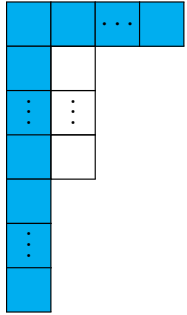
\[\begin{ytableau} *(cyan)& *(cyan)& *(cyan)\cdots & *(cyan) \\ *(cyan) & \\ *(cyan)\raisebox{-0.3ex}{\vdots} & \raisebox{-0.3ex}{\vdots} \\ *(cyan) & \\ *(cyan) \\ *(cyan)\raisebox{-0.3ex}{\vdots} \\ *(cyan) \\ \end{ytableau}\]
\item We now assume that $\lambda_{1}-\kappa_{1}=1$, so that $\lambda_{1}=i+1\geq2$, and we claim that $u=r-i$. Indeed, when this is not the case we get $\lambda_{r-i+1}=1=\kappa_{r-i+1}$, so that we can remove the box in that row from $\kappa$ and replace it by the remaining one in the first row and get an element as in Lemma \ref{charcyc} whose complement in $\lambda$ is in $\operatorname{VS}_{\lambda}$, which cannot be under our assumption. It follows that $\lambda\in\Gamma^{r}_{k}$ via the third case in Definition \ref{rprim}, with the sign appearing in the constant term being $(-1)^{i}=(-1)^{r-u}=\varepsilon^{r}_{\lambda}$ as required. The picture of $\lambda$, with $\kappa$ in cyan inside it, here look as follows. 
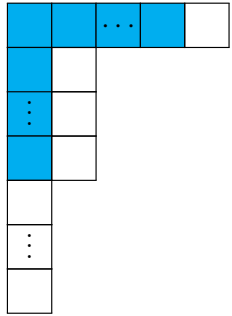
\[\begin{ytableau} *(cyan)& *(cyan)& *(cyan)\cdots & *(cyan) & \\ *(cyan) & \\ *(cyan)\raisebox{-0.3ex}{\vdots}& \\ *(cyan) & \\ \\ \raisebox{-0.3ex}{\vdots} \\ \\ \end{ytableau}\]
\end{enumerate}
We therefore showed that the only case where a non-zero constant term exists is when $\lambda\in\Gamma^{r}_{k}$, and in every such case this constant term is the sign $\varepsilon^{r}_{\lambda}$. This completes the proof of the lemma.
\end{proof}
The proof of Theorem \ref{main} in this method now continues as in the previous one.

\smallskip

We conclude by remarking about partitions that are not single cycles. The inductive construction will work in general, but the issue of determining the basic cases will be more complicated. This is visible in the first proof of Lemma \ref{genbasis}, where the application of the Murnaghan--Nakayama Rule, producing characters values of partitions with one cycle less, are no longer just the dimensions. In the second proof, via Proposition \ref{recpart} and Corollary \ref{specs}, the binomial coefficients $\prod_{i=2}^{n}\binom{x_{i}(\sigma)}{a_{i}}$ will be less restrictive, thus producing more options of $\pi_{\alpha}$ and with them more partitions $\kappa$ that allow for non-zero contributions. In fact, in the few cases that we checked, even for the simplest non-cycle permutation of cycle type $(2,2,1^{n-4})$ we encounter constant terms that are no longer sign, but can equal 2 or $-2$. Hence the analysis of other permutations is more complicated, and is left for future research.

\medskip

\noindent\textsc{Einstein Institute of Mathematics, the Hebrew University of Jerusalem, Edmund Safra Campus, Jerusalem 91904, Israel}

\noindent E-mail address: tom.moshaiov@mail.huji.ac.il

\noindent\textsc{Einstein Institute of Mathematics, the Hebrew University of Jerusalem, Edmund Safra Campus, Jerusalem 91904, Israel}

\noindent E-mail address: zemels@math.huji.ac.il

\end{document}